\documentclass[12pt]{amsart}
\usepackage{amsmath}
\usepackage{amssymb}
\usepackage{amsthm}
\usepackage{enumerate}
\usepackage[mathscr]{eucal}
\usepackage{color}
\usepackage{graphics}
\theoremstyle{plain}
\newtheorem{theorem}{Theorem}[section]

\newtheorem{lemma}{Lemma}[section]

\theoremstyle{definition}

\newtheorem{remark}{Remark}[section]

\newtheorem{example}{Example}[theorem]
\newtheorem*{acknowledgement}{\textnormal{\textbf{Acknowledgements}}}
\usepackage[pagewise]{lineno}
\begin{document}
\title[On approximate Birkhoff-James orthogonality and normal cones]{ On approximate Birkhoff-James orthogonality and normal cones in a normed space}
\author[Debmalya Sain, Kallol Paul and Arpita Mal]{Debmalya Sain, Kallol Paul and Arpita Mal}

\newcommand{\acr}{\newline\indent}

\address[Sain]{Department of Mathematics\\ Indian Institute of Science\\ Bengaluru 560012\\ Karnataka \\India\\ }
\email{saindebmalya@gmail.com}

\address[Paul]{Department of Mathematics\\ Jadavpur University\\ Kolkata 700032\\ West Bengal\\ INDIA}
\email{kalloldada@gmail.com}

\address[Mal]{Department of Mathematics\\ Jadavpur University\\ Kolkata 700032\\ West Bengal\\ INDIA}
\email{arpitamalju@gmail.com}

\thanks{ The research of the first author is sponsored by Dr. D. S. Kothari Postdoctoral fellowship. The third author would like to thank UGC, Govt. of India for the financial support.} 


\begin{abstract}
We study two notions of approximate Birkhoff-James orthogonality in a normed space, from a geometric point of view, and characterize them in terms of normal cones. We further explore the interconnection between normal cones and approximate Birkhoff-James orthogonality to obtain a complete characterization of normal cones in a two-dimensional smooth Banach space. We also obtain a uniqueness theorem for approximate Birkhoff-James orthogonality set in a normed space.

\end{abstract}

\maketitle

\section{Introduction.} 

Birkhoff-James orthogonality is arguably the most natural and well-studied notion of orthogonality in a normed space. Indeed, in the normed space setting, the close connection shared by Birkhoff-James orthogonality with various geometric properties like strict convexity, smoothness etc. can hardly be over emphasized. As a consequence, Birkhoff-James orthogonality plays a crucial role in exploring the geometry of normed spaces \cite{J}. In view of this, it is perhaps not surprising that various generalizations (and approximations) of Birkhoff-James orthogonality have been introduced and studied by several mathematicians, including Dragomir \cite{D} and Chmieli\'nski \cite{C}. In this paper, our aim is to study two different approximations of Birkhoff-James orthogonality, in order to have a better understanding of the geometry of normed spaces. Among other things, we exhibit that both types of approximate Birkhoff-James orthogonality have a close connection with normal cones in a normed space. Without further ado, let us establish our notations and terminologies to be used throughout this paper.\\
Let $\mathbb{X}$ be a normed space defined over $ \mathbb{R}, $ the field of real numbers. Let $ B_{\mathbb{X}} $ and $ S_{\mathbb{X}} $ denote the unit ball and the unit sphere of $ \mathbb{X} $ respectively, i.e.,  $ B_\mathbb{X}=\{x \in \mathbb{X} : \|x\| \leq 1\} $ and $ S_\mathbb{X}=\{x \in \mathbb{X} : \|x\|=1\}. $ For any two elements $x$ and $y$ in $\mathbb{X}$, $x$ is said to be orthogonal to $y$ in the sense of Birkhoff-James \cite{B} if $\|x + \lambda y \| \geq \|x\|$ for all real scalars $\lambda.$ Recently, Sain \cite{S} introduced the notions of $x^+,~x^-$ for a given element $ x $ of $ \mathbb{X}, $ in order to completely characterize Birkhoff-James orthogonality of linear operators on a finite dimensional Banach space. An element $y \in \mathbb{X}$ is said to be in $x^+$ if $\|x+ \lambda y\| \geq \|x\|$ for all $\lambda \geq 0$ and $y \in \mathbb{X}$ is said to be in $x^-$ if $\|x+ \lambda y\| \geq \|x\|$ for all $\lambda \leq 0$. 
The notion of Birkhoff-James orthogonality has been generalized by Dragomir \cite{D} in the following way, in order to obtain a suitable definition of approximate Birkhoff-James orthogonality in normed spaces.\\
Let $ \epsilon \in [0,1) .$ Then for $x, y \in \mathbb{X}$, $x$ is said to be approximate $ \epsilon- $ Birkhoff-James orthogonal to $y$  if 
\[ \|x + \lambda y\| \geq (1-\epsilon) \|x\| ~\forall \lambda \in \mathbb{R} .\] 

Later on, Chmieli\'nski \cite{C} slightly modified the definition in the following way: \\
Let $ \epsilon \in [0,1).$ Then for $x, y \in \mathbb{X}$, $x$ is said to be approximate $ \epsilon- $ Birkhoff-James orthogonal to $y$  if 
\[ \|x + \lambda y\| \geq \sqrt{1-\epsilon^2} \|x\| ~\forall \lambda \in \mathbb{R} .\]
In this case, we write $ x \bot_D^{\epsilon} y.$ \\

Chmieli\'nski \cite{C} also introduced another notion of approximate Birkhoff-James orthogonality, defined in the following way: \\
Let $ \epsilon \in [0,1) .$ Then for $x, y \in \mathbb{X}$, $x$ is said to be approximate $ \epsilon- $ Birkhoff-James orthogonal to $y$  if 
\[ \|x + \lambda y\|^2 \geq  \|x\|^2 - 2 \epsilon \|x\| \| \lambda y\| ~\forall \lambda \in \mathbb{R} .\] 
In this case, we write $ x \bot_B^{\epsilon} y.$ \\

It should be noted that in an inner product space, both types of approximate Birkhoff-James orthogonality coincide. However, this is not necessarily true in a normed space. We would also like to remark that in a normed space, both types of approximate Birkhoff-James orthogonality are homogeneous. \\
Recently Chmieli\'nski et al. \cite{CSW} characterized ``$ x \bot_B^{\epsilon} y$'' for real normed spaces by means of the following theorem: \\
\begin{theorem} [Theorem 2.3,\cite{CSW}]
Let $\mathbb{X}$ be a real normed space. For $ x, y \in \mathbb{X} $ and $ \epsilon \in [0,1) :$ 
\[ x \bot_B^{\epsilon} y \Leftrightarrow \exists ~z \in ~\mbox{Lin}\{x,y\} : x \bot_B z, \|z-y\| \leq \epsilon \|y\|.\]
\end{theorem} 

The above characterization of ``$ x \bot_B^{\epsilon} y$'' is essentially analytic in nature. In this paper, our aim is to explore the structure and properties of both types of approximate orthogonal vectors from a geometric point of view. In order to serve this purpose, we introduce the following two notations: \\
Given $x \in \mathbb{X} $ and $\epsilon \in [0,1), $ we define  
\[ F(x, \epsilon) = \{ y \in \mathbb{X} :  x \bot_D^{\epsilon} y \}.\]
\[ G(x, \epsilon) = \{ y \in \mathbb{X} :  x \bot_B^{\epsilon} y \}.\]
In this context, the concept of normal cones in a normed space play a very important role. Therefore, at this point of our discussion, the following definition is in order: \\
A subset $ K $ of $ \mathbb{X} $ is said to be a normal cone in $ \mathbb{X} $ if \\
$ (i)~ K + K \subset K, (ii)~ \alpha K \subset K $ for all $ \alpha \geq 0 $ and $ (iii)~ K \cap (-K) = \{\theta\}. $
Normal cones are important in the study of geometry of normed spaces because there is a natural partial ordering $ \geq $ associated with $ K, $ namely, for any two elements $ x, y \in \mathbb{X},~ x \geq y $ if $ x-y \in K. $ It is easy to observe that in a two-dimensional Banach space $ \mathbb{X}, $ any normal cone $ K $ is completely determined by the intersection of $ K $ with the unit sphere $ S_{\mathbb{X}}. $  Keeping this in mind, when we say that $ K $ is a normal cone in $ \mathbb{X}, $ determined by $ v_1, v_2, $ what we really mean is that $ K \cap S_{\mathbb{X}} = \{\frac{(1-t)v_1 + tv_2}{\| (1-t)v_1 + tv_2 \|} : t \in[0,1]\}. $ Of course, in this case $ K = \{\alpha v_1 + \beta v_2 : \alpha, \beta \geq 0\}. $ We prove that in a two-dimensional Banach space $ \mathbb{X}, $ given any $ x \in \mathbb{X} $ and any $ \epsilon \in [0, 1), $ \\
$ F(x, \epsilon) = K_1 \cup (-K_1) $ and $ G(x, \epsilon) = K_2 \cup (-K_2), $ where $ K_1, K_2 $ are normal cones in $ \mathbb{X}. $ In fact, in case of $ F(x, \epsilon), $ we prove something more. We show that in a two-dimensional smooth Banach space $ \mathbb{X}, $ given any normal cone $ K, $ there exists $ x \in S_{\mathbb{X}} $ and $ \epsilon \in [0, 1) $ such that $ F(x, \epsilon) = K \cup (-K). $ Indeed, this interconnection between normal cones and approximate $ \epsilon- $ Birkhoff-James orthogonality sets is a major highlight of the present paper. Equipped with the above mentioned interconnection, we proceed to obtain a complete geometric description of $ F(x, \epsilon) $ and $ G(x, \epsilon), $ when $ \mathbb{X} $ is a normed space of any dimension. In order to accomplish this goal, let us introduce the following two notations: \\
Let $\mathbb{X}$ be a normed space. For $x, y \in \mathbb{X}$ and $\epsilon \in [0,1),$ let $P_{x, y}(\epsilon), Q_{x, y}(\epsilon) $ denote the restriction of $ F(x, \epsilon) $ and $ G(x, \epsilon) $ respectively to the subspace spanned by $x$ and $y$. Using these notations, we obtain a complete geometric description of  $ F(x, \epsilon) $ and $ G(x, \epsilon), $ when $ \mathbb{X} $ is any normed space. We show that both $ F(x, \epsilon) $ and $ G(x, \epsilon) $ are union of two-dimensional normal cones. We also prove  a uniqueness theorem for $ F(x, \epsilon), $ first in the case of a two-dimensional Banach space and then for any normed space.

\section{Main results.}

\begin{theorem}\label{theorem:F(x,epsilon)}
Let $\mathbb{X}$ be a two-dimensional Banach space. Then for any $x \in S_{\mathbb{X}}$ and $\epsilon \in [0, 1),$ there exists a normal cone $K$ in $\mathbb{X}$ such that $F(x, \epsilon)= K \cup (-K)$.
\end{theorem}
\begin{proof}
For $x\in S_{\mathbb{X}}$ there exists $y \in S_{\mathbb{X}}$ such that $x \perp_B y$. Consider 
\[A= \{t\in [0,1]: x \perp_D^{\epsilon}\{(1-t)x+ t y\}\},\]
\[B= \{t\in [0,1]: x \perp_D^{\epsilon}\{-(1-t)x+ t y\}\}\]
\begin{figure}[ht]
\centering 
\includegraphics{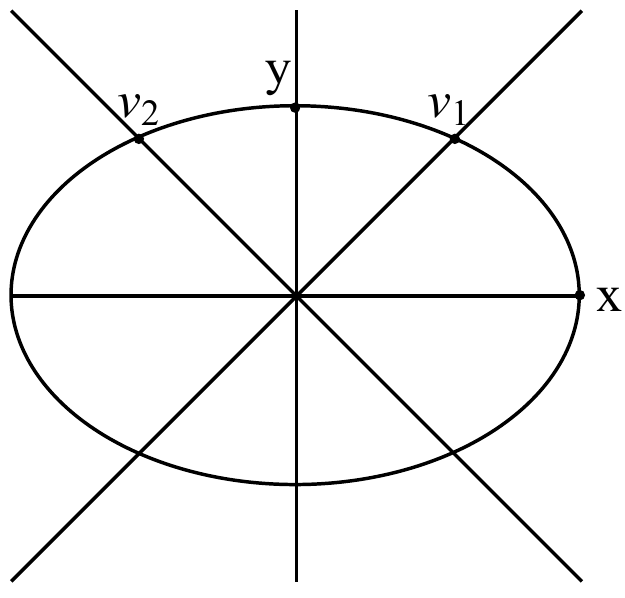}
\caption{}
\label{Figure 1}
\end{figure}

Clearly, both $A$ and $B$ are non empty proper subset of $[0,1],$ as $0\notin A \cup B$ and $1 \in A \cap B$. We next show that both $A$ and $B$ are closed. Suppose $\{t_n\}\subseteq A$ be such that $t_n \longrightarrow t$. Then $x \perp_D^{\epsilon}\{(1-t_n)x+ t_n y\} \Longrightarrow \|x + \lambda \{(1-t_n)x + t_n y\}\| \geq \sqrt{1-{\epsilon}^2}$ for all $\lambda \in \mathbb{R}$. Letting $n \longrightarrow \infty$ we have,  $\|x + \lambda \{(1-t)x + t y\}\| \geq \sqrt{1-{\epsilon}^2} $ for all $\lambda \in \mathbb{R}$. Then  $x \perp_D^{\epsilon} \{(1-t)x+ty\}$ and so $t \in A.$ This proves that $ A $ is closed. Similarly, one can show that $B$ is closed. \\
Let $t_1= \inf A,~ t_2= \inf B$. Clearly $ t_1 \neq 0,~ t_2 \neq 0.$ \\
Let $u_1= (1-t_1)x+t_1 y$, $u_2= -(1-t_2)x+t_2 y$ and $v_1=\frac{u_1}{\|u_1\|}$, $v_2=\frac{u_2}{\|u_2\|}$. Let $K$ be the normal cone determined by $v_1, v_2$. We next show that $F(x, \epsilon)= K \cup (-K)$. \\
If $z\in \mathbb{X}$ is such that $z=c\{(1-t)x+t y\}$, where $c> 0$ and $t< t_1$ then from the definition of infimum, $x \not \perp_D^{\epsilon} z$. Let $w \in \mathbb{X}$ be such that $w= c \{(1-t)x + t y\}$, where $c \geq 0$ and $1 \geq t \geq t_1$. We show that $x \perp_D^{\epsilon} w$. Let $\lambda > 0$. Choose $\eta_1=1+ \lambda c  (1-t) \geq 1$ and choose $\lambda_1=\frac{\lambda c t}{\eta_1}$. Then $x+ \lambda w = \eta_1(x+ \lambda_1 y)$. Hence
\begin{eqnarray*}
\|x+ \lambda w\|&=& |\eta_1|\|x+ \lambda_1 y\|\\
                &\geq & \|x+ \lambda_1 y\|\\
                & \geq & 1\\
                & \geq & \sqrt{1-{\epsilon}^2}
\end{eqnarray*}

\begin{figure}[ht]
\centering 
\includegraphics{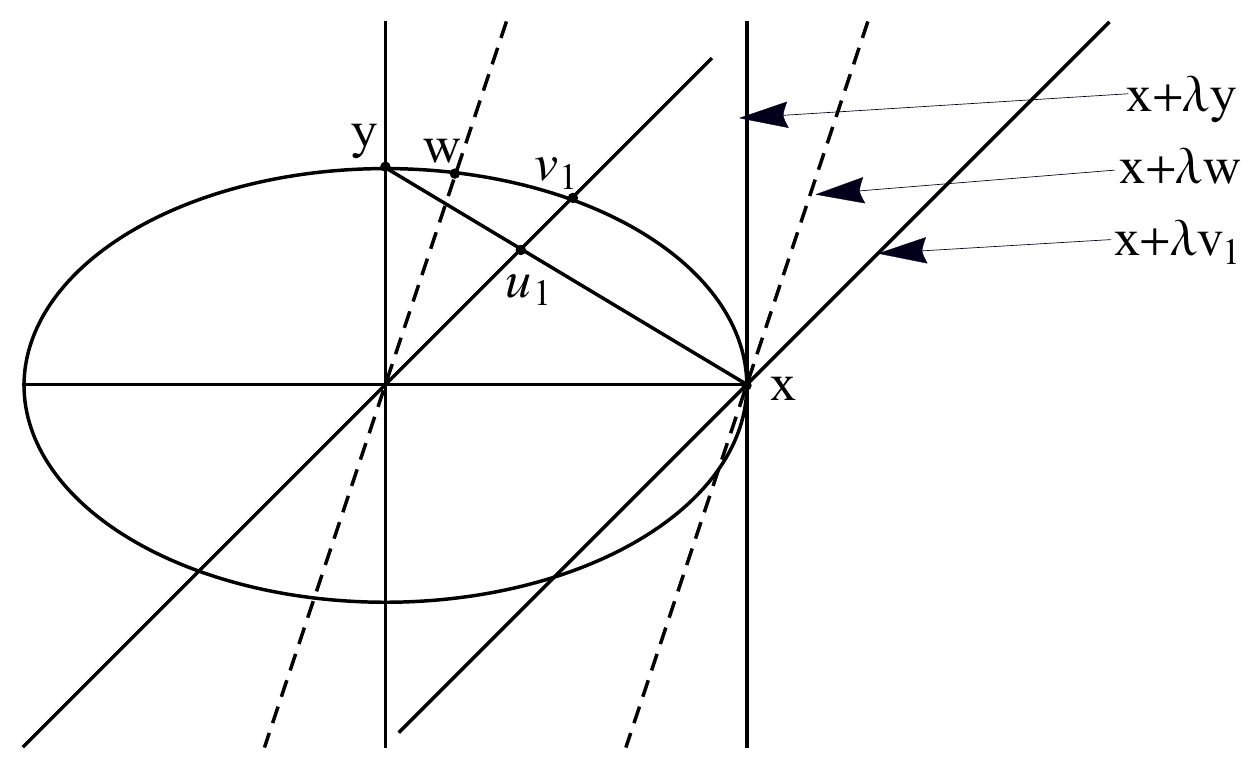}
\caption{}
\label{Figure 1}
\end{figure}
Let $\lambda < 0$. Write $w=b \{(1-s)x + s v_1\}$, where $b \geq 0$, $s \geq 1$. Choose $\eta_2=1+ \lambda b (1-s) \geq 1$ and $\lambda_2=\frac{\lambda b s}{\eta_2}$. Then $x+ \lambda w= \eta_2 (x+ \lambda_2 v_1)$. Hence  
\begin{eqnarray*}
\|x+ \lambda w\|&=&|\eta_2|\|x+ \lambda_2 v_1\|\\
                &\geq & \|x+ \lambda_2 v_1\|\\
                &\geq& \sqrt{1-{\epsilon}^2}
\end{eqnarray*}
Thus, for all $\lambda \in \mathbb{R}$, $\|x+ \lambda w\|\geq \sqrt{1-{\epsilon}^2} \Rightarrow x \perp_D^{\epsilon}w$. Similarly, if $z=c\{-(1-t)x+t y\}$, where $c> 0$ and $t< t_2$ then $x \not \perp_D^{\epsilon} z$ and if $w=d\{-(1-t)x+t y\}$, where $d\geq 0$ and $1 \geq t\geq t_2$ then $x \perp_D^{\epsilon} w$. Therefore, $F(x, \epsilon)= K \cup (-K)$. 
\end{proof}

\begin{remark}
We further observe that if $\epsilon \in (0,1)$ then $ t_1 < 1$, for if      
 $z \in S_{\mathbb{X}}$ and $\|z-y\| \leq \frac{1- \sqrt{1-{\epsilon}^2}}{1+ \sqrt{1-{\epsilon}^2}}$ then $\|x+ \lambda z\| \geq |\|\lambda z\|- \|x\||\geq |\lambda|-1 > \sqrt{1-{\epsilon}^2},$ whenever $|\lambda| > 1+ \sqrt{1-{\epsilon}^2}$. Also, if $|\lambda| \leq 1+\sqrt{1-{\epsilon}^2}$ then 
\begin{eqnarray*}
 \|x + \lambda z\| & = & \|x+ \lambda y + \lambda (z-y)\| \\
                   & \geq &\|x + \lambda y\| - |\lambda|\|z-y\| \\
									& \geq & 1- (1+\sqrt{1-{\epsilon}^2})\|z-y\| \\
									& \geq & \sqrt{1-{\epsilon}^2},
\end{eqnarray*}	
so that $x \perp_D^{\epsilon} z.$ This shows that there exists some  $t,~0< t < 1 $ such that $x \perp_D^{\epsilon} \{(1-t)x+ty\}$. Similarly, if $\epsilon \in (0,1)$ then  $ t_2 < 1$. 
\end{remark}

\begin{remark}\label{remark:plusminus}
From the proof of the Theorem \ref{theorem:F(x,epsilon)}, it is clear that if $ \epsilon > 0 $ then both $v_1,~v_2$ can not be simultaneously in $x^+$ (or $x^-)$. In fact, for $ \epsilon > 0, $ exactly one of $ v_1, v_2 $ will be in $ x^{+} $ and the other one will be in $ x^{-}. $ On the other hand, it is easy to observe that in two-dimensional smooth Banach spaces, $ v_1 = v_2, $ if $ \epsilon = 0. $
\end{remark}

Next, let us consider the set $S(x, \epsilon)=\{z\in S_{\mathbb{X}}: ~\inf_{\lambda \in \mathbb{R}}\|x+ \lambda z\|= \sqrt{1- {\epsilon}^2}\}$. In context of the previous theorem, it is possible to obtain a nice characterization of $ S(x, \epsilon). $ We accomplish the goal in the following theorem:

\begin{theorem}\label{theorem:solution}
Let $\mathbb{X}$ be a two-dimensional Banach space and $v_1, v_2$ be as in Theorem \ref{theorem:F(x,epsilon)}. Then $S(x, \epsilon)= \{\pm v_1, \pm v_2\}$, if $\epsilon \in (0,1)$ and $S(x, \epsilon)= \{y\in S_{\mathbb{X}} :~ x \perp_B y\}$ if $\epsilon=0$. 
\end{theorem}
\begin{proof}
First, let $\epsilon \in (0,1)$. Note that every element in $S_{\mathbb{X}}$ can be written as some  multiple of a convex combination of $v_2,y$ or $y ,v_1$ or $v_1, x$ or $x,-v_2$. From Theorem \ref{theorem:F(x,epsilon)}, it is clear that $\inf_{\lambda}\|x+ \lambda z\|< \sqrt{1-{\epsilon}^2}$ if $z $ is some  multiple of a convex combination of $v_1,x$ or $x,-v_2$.  Now, suppose $z= \frac{(1-t)y + t v_1}{\|(1-t)y+ t v_1\|}$ for some $0\leq t < 1$. Let $\lambda < 0$. Write $z=a \{(1-s)x+ s v_1\}$, where $s>1$ and $a= \frac{1}{\|(1-s)x+ s v_1\|}> 0$. Let $\eta= 1+ a \lambda (1-s)> 1$ and $\lambda ' = \frac{a \lambda s}{\eta}$. Then $x+\lambda z = \eta\{x+ \lambda ' v_1\} \Rightarrow \|x+ \lambda z\|> \|x + \lambda' v_1\|\geq \sqrt{1-{\epsilon}^2}$. Now, let $\lambda \geq 0$. Write $z=b\{(1-s')x+s' y\}$, where $0< s'\leq 1$ and $b=\frac{1}{\|(1-s')x+ s' y\|}$. Let $\eta_1= 1+ \lambda b (1-s') \geq 1$ and $\lambda_1= \frac{\lambda b s'}{\eta_1}$. Then $x+\lambda z= \eta_1\{x+ \lambda_1 y\} \Rightarrow \|x+ \lambda z\|\geq \|x+ \lambda_1 y\|\geq 1 > \sqrt{1-{\epsilon}^2}$. Therefore, $\inf_{\lambda \in \mathbb{R}} \|x+ \lambda z\| > \sqrt{1-{\epsilon}^2}$. Similarly,  if $z=\frac{(1-t)y + t v_2}{\|(1-t)y + t v_2\|}$ for some $0\leq t< 1$ then we can show that   $\inf_{\lambda}\|x+ \lambda z\|> \sqrt{1-{\epsilon}^2}$. Thus, $ z \in S(x, \epsilon)$  if and only if  $z \in  \{\pm v_1, \pm v_2\} $.\\
The second part of the theorem is obvious.
\end{proof}

Let $P_{x, y}(\epsilon)$ denote the restriction of $ F(x, \epsilon) $  to the subspace spanned by $x$ and $y$.  Equipped with the characterization of $ F(x, \epsilon) $ in two-dimensional Banach spaces, it is now possible to obtain a complete description of $ F(x, \epsilon) $ in any normed space.   The following theorem, the proof of which is immediate, illustrates our claim.

\begin{theorem}\label{theorem:any dimension}
Let $\mathbb{X}$ be a normed space. Let $x \in \mathbb{X}$ and $\epsilon \in [0,1)$. Then $F(x, \epsilon)= \bigcup_{y\in \mathbb{X}}P_{x,y}(\epsilon)$. In particular, $ F(x, \epsilon) $ is a union of two-dimensional normal cones.
\end{theorem}

Let us now turn our attention to the converse of Theorem $ 2.1. $ As promised in the introduction, we prove that for any normal cone $ K $ in a two-dimensional smooth Banach space $ \mathbb{X}, $ there exists some $ x \in S_{\mathbb{X}} $ and some $ \epsilon \in [0,1) $ such that $ K \cup (-K)= F(x, \epsilon). $

\begin{theorem}\label{theorem:cone3}
Let $\mathbb{X}$ be a two-dimensional smooth Banach space. Let $K$ be a normal cone in $\mathbb{X}$. Then there exists $x \in S_{\mathbb{X}}$ and $\epsilon \in [0,1)$ such that $F(x, \epsilon)= K \cup (-K)$.
\end{theorem}
\begin{proof}
Let the cone $K$ be determined by $v_1, ~v_2 \in S_{\mathbb{X}}$. Note that since $ K $ is a normal cone, $ v_1 \neq -v_2. $ First, let us assume that $ v_1 \neq v_2. $ Consider the two sets
\[W_1= \{x\in S_{\mathbb{X}}:~ \inf_{\lambda}\|x+ \lambda v_1\|> \inf_{\lambda}\|x+ \lambda v_2\|\}, \]
\[W_2= \{x\in S_{\mathbb{X}}:~ \inf_{\lambda}\|x+ \lambda v_2\|> \inf_{\lambda}\|x+ \lambda v_1\|\}. \]
Clearly, $W_1 \neq \emptyset$ and $W_2 \neq \emptyset$, since $v_2 \in W_1$ and $v_1 \in W_2$. We show that $W_1$ is open. Let $z \in W_1$. Let $l_1= \inf_{\lambda}\|z+ \lambda v_1\|> \inf_{\lambda}\|z+ \lambda v_2\|= l_2$. Choose $\epsilon' >0$ such that $\epsilon' < \frac{l_1- l_2}{2}$. We claim that $B(z, \epsilon') \cap S_{\mathbb{X}} \subseteq W_1$. By standard compactness argument, there exists $\lambda_0$ such that $\|z+ \lambda_0 v_2\|= l_2$. Now, suppose that $w \in B(z, \epsilon') \cap S_{\mathbb{X}}$. Then for any  $\lambda \in \mathbb{R}$, we have, 
\begin{eqnarray*}
            \|w +\lambda v_1\| & = & \|(z+\lambda v_1)-(z- w)\| \\
							                 & \geq & \|z+\lambda v_1\|- \|z-w\| \\
                                &> & l_1- \epsilon'
\end{eqnarray*}
Also,  
\begin{eqnarray*}
            \|w+ \lambda_0 v_2\| & = &\|(z+ \lambda_0 v_2)+(w-z)\| \\
                                 & \leq & \|z+ \lambda_0 v_2\| +\|w-z\| \\
																& < & l_2+ \epsilon' 
\end{eqnarray*}																
Therefore, $\inf_{\lambda}\|w+\lambda v_1\| > l_1 -\epsilon' > l_2+ \epsilon' > \inf_{\lambda} \|w+\lambda v_2\|$. This implies that $w \in W_1$. Hence, $B(z, \epsilon') \cap S_{\mathbb{X}} \subseteq W_1$. Therefore, $W_1$ is open in $S_{\mathbb{X}}$. Similarly it can be shown that $W_2$ is open. Again, $W_1 \cap W_2 = \emptyset$. Since $S_{\mathbb{X}}$ is connected, $W_1 \cup W_2 \neq S_{\mathbb{X}}$. Choose $x \in S_{\mathbb{X}}\setminus (W_1 \cup W_2)$. Then clearly $\inf_{\lambda}\|x+ \lambda v_1\|= \inf_{\lambda}\|x+ \lambda v_2\|$. We would like to remark that since $\mathbb{X}$ is smooth, and $ v_1 \neq \pm v_2, $ $ x $ can not be Birkhoff-James orthogonal to both $ v_1 $ and $ v_2. $ Since $\inf_{\lambda}\|x+ \lambda v_1\|= \inf_{\lambda}\|x+ \lambda v_2\|,$ it is now easy to see that $ x $ is not Birkhoff-James orthogonal to either of $ v_1, v_2. $ Therefore, there exists $\epsilon \in (0,1)$ such that $\inf_{\lambda}\|x+ \lambda v_1\|= \inf_{\lambda}\|x+ \lambda v_2\|= \sqrt{1-{\epsilon}^2}$. Now, by Theorem \ref{theorem:F(x,epsilon)} and Theorem \ref{theorem:solution}, it is clear that $F(x, \epsilon)=K \cup (-K)$.\\
Let us now consider the case $ v_1 = v_2. $ Clearly, $ K $ is simply a half-line in this case, given by $ K = \{\lambda v_1 : \lambda \geq 0\}. $ By Theorem 2.3 of James \cite{J}, there exists $ x \in S_{\mathbb{X}} $ such that $ x \perp_{B} v_1(=v_2). $ Furthermore, since $ \mathbb{X} $ is smooth, if $ y \in S_{\mathbb{X}} $ is such that $ x \perp_{B} y, $ then $ y = \pm v_1. $ Therefore, by choosing $ \epsilon = 0, $ it is now immediate that $F(x, \epsilon)=K \cup (-K)$. This establishes the theorem.
\end{proof}
 
In the following example, we illustrate the fact that the smoothness assumption in Theorem $ 2.4 $ can not be dropped.

\begin{example}
Let $\mathbb{X}= l_{\infty}^2$. Let $K$ be the normal cone determined by $(-\frac{1}{2},1)$ and $(-1,1)$. If possible, suppose that there exists $x \in S_{\mathbb{X}}$ and $\epsilon \in [0,1)$ such that $F(x, \epsilon)= K \cup (-K)$. Let $x \perp_B z$ and $z\in S_{\mathbb{X}}$. Then $z\in K \cup (-K)$. Without loss of generality assume that $z \in K$. Then $z= (1-t)(-\frac{1}{2},1)+ t (-1,1)$ for some $t \in [0,1]$. Then clearly $x= \pm (1,1)$. Since $ x \bot_B (-1,0)$, we must have, $(-1,0)\in F(x, \epsilon)$. But $(-1,0) \notin K \cup (-K)$. This contradiction proves that there does not exist any $x \in S_{\mathbb{X}}$ and $\epsilon \in [0,1)$ such that $F(x, \epsilon)= K \cup (-K)$.
\end{example}

In the next theorem, we prove a uniqueness result for $ F(x, \epsilon) $ in a two-dimensional Banach space for $ \epsilon \in (0,1).$ 

\begin{theorem}\label{theorem:F(x_1,e_1)=F(x_2,e_2)}
Let $\mathbb{X}$ be a two-dimensional Banach space. Let $x_1, x_2 \in S_{\mathbb{X}}$ and $\epsilon_1, \epsilon_2 \in (0,1)$. If $F(x_1, \epsilon_1)= F(x_2, \epsilon_2)$ then $x_1 = \pm x_2$ and $\epsilon_1= \epsilon_2$.
\end{theorem}
\begin{proof}
We prove the theorem in the following two steps. \\
\noindent \textbf{Step 1:} If $F(x,\epsilon)=F(y,\epsilon)$ for some $ x, y \in S_{\mathbb{X}}, $  then $x=\pm y$. \\
Let $F(x,\epsilon)=F(y,\epsilon)=K \cup (-K),$ where $K$ is the normal cone determined by $v_1,v_2 \in S_{\mathbb{X}}$. It is clear from the proof of the Theorem \ref{theorem:F(x,epsilon)} that $x$ is either in the cone determined by $v_1,-v_2$ or in the cone determined by $-v_1,v_2$. Moreover,  the same is true for $y$ also. Assume that $x$ is in the cone determined by $v_1,-v_2$ so that $x=a\{(1-t_1)v_1 - t_1 v_2\}$ for some $t_1 \in (0,1)$ and $a >0$. Suppose $y$ is in the cone determined by $v_1,-v_2$ and $y=b\{(1-t)x - t v_2\}$ where $b> 0$ and $0 \leq t \leq 1$. Clearly $t \neq 1$. If possible, suppose that $t \neq 0$. Let $(x+ \beta v_2) \perp_B v_2$ and $(y+ \mu v_1)\perp_B v_1$. Then from the Theorem \ref{theorem:solution}, it is clear that $\|x+ \beta v_2\|=\|y+ \mu v_1\|=\sqrt{1-{\epsilon}^2}$. Clearly, $x+ \beta v_2= k(y + \lambda v_2)$ for some scalars $k, \lambda$. Thus, $x + \beta v_2= k b(1-t)x -kbt v_2 + k \lambda v_2 \Rightarrow k= \frac{1}{b(1-t)}$. If $k> 1$ then $\sqrt{1-{\epsilon}^2}=\|x+ \beta v_2\|=|k| \|y+ \lambda v_2\|> \|y+ \lambda v_2\|\geq \sqrt{1-{\epsilon}^2}$, a contradiction. Therefore, $ k= \frac{1}{b(1-t)}\leq 1$. Again, $y+ \mu v_1= k_1(x+ \lambda_1 v_1)$ for some scalar $k_1, \lambda_1$. Therefore, $b(1-t)x-b t[\frac{1-t_1}{t_1}v_1 - \frac{1}{at_1}x]+ \mu v_1 = k_1 x + k_1 \lambda_1 v_1 \Rightarrow k_1= b(1-t)+ \frac{bt}{at_1} \Rightarrow k_1 > 1$. Therefore, $\sqrt{1-{\epsilon}^2}= \|y + \mu v_1\|= |k_1|\|x + \lambda_1 v_1\|> \|x+ \lambda_1 v_1\|\geq \sqrt{1-{\epsilon}^2}$, a contradiction. Hence we must have $t=0$. However, since $ x, y \in S_{\mathbb{X}}, $ this implies that $ x=y. $ Similarly, $y= b\{(1-t)x +tv_1\}$ for some $b>0$ implies that $t=0$ and once again we have, $y=x$. On the other hand,  if $y$ is in the cone determined by $-v_1,v_2$ then it can be shown using similar arguments that $x = - y.$\\
\noindent \textbf{Step 2:} Claim $F(x_1, \epsilon_1)= F(x_2,\epsilon_2)\Rightarrow \epsilon_1= \epsilon_2$ . \\
Let $F(x_1, \epsilon_1)= F(x_2,\epsilon_2)= K \cup (-K)$, where $K$ is the normal cone determined by $v_1,v_2 \in S_{\mathbb{X}}$. Then by Theorem \ref{theorem:solution}, it is clear that $\inf_{\lambda \in \mathbb{R}}\|x_2 + \lambda v_1\|=\inf_{\lambda \in \mathbb{R}}\|x_2 + \lambda v_2\|=\sqrt{1-{\epsilon_2}^2}$ and $\inf_{\lambda \in \mathbb{R}}\|x_1 + \lambda v_1\|= \inf_{\lambda \in \mathbb{R}}\|x_1 + \lambda v_2\|=\sqrt{1-{\epsilon_1}^2}$. Therefore, if $x_1= \pm x_2$ then $\epsilon_1= \epsilon_2$ and we are done. Hence assume that $x_1 \neq \pm x_2$ . Clearly $x_1,~x_2 \notin K \cup (-K)$.
From remark \ref{remark:plusminus}, we have,  either $v_1 \in {x_1}^+, ~ v_2 \in {x_1}^- $ or $v_1 \in {x_1}^-, ~v_2 \in {x_1}^+$. Without loss of generality, assume that $v_1 \in {x_1}^+, ~v_2 \in {x_1}^- $. First suppose that $x_2= \frac{(1-t)x_1+ t v_1}{\|(1-t)x_1+ t v_1\|}$ for some $0<t<1$. Then it can be easily verified that $v_1 \in {x_2}^+, ~v_2 \in {x_2}^- $. Now, $1= \|\frac{(1-t)x_1+ t v_1}{\|(1-t)x_1+ t v_1\|}\|=\frac{1-t}{\|(1-t)x_1+ t v_1\|}\|x_1+\frac{t}{1-t}v_1\| \geq \frac{1-t}{\|(1-t)x_1+ t v_1\|}$, since $v_1 \in {x_1}^+ $. Now for any $\lambda \in \mathbb{R}$, $\|x_2+ \lambda v_1\|= \|\frac{(1-t)x_1+ t v_1}{\|(1-t)x_1+ t v_1\|} + \lambda v_1\|=\frac{1-t}{\|(1-t)x_1+ t v_1\|}\|x_1+ (\frac{t}{1-t}+ \frac{\lambda \|(1-t)x_1+ t v_1\|}{1-t} )v_1\| \Rightarrow \inf_{\lambda \in \mathbb{R}}\|x_2 + \lambda v_1\|= \frac{1-t}{\|(1-t)x_1+ t v_1\|}\inf_{\lambda \in \mathbb{R}}\|x_1 + \lambda v_1\| \leq \inf_{\lambda \in \mathbb{R}}\|x_1 + \lambda v_1\| \Rightarrow \sqrt{1- {\epsilon_2}^2} \leq \sqrt{1- {\epsilon_1}^2}$.  \\
We can write $x_2 = \frac{(1-t)x_1+ t v_2}{\|(1-t)x_1+ t v_2\|}$ for some $0< t < 1$. This implies that $x_1= \frac{\|(1-t)x_1+ t v_2\|}{1-t}x_2- \frac{t}{1-t}v_2 \Rightarrow 1= \frac{\|(1-t)x_1+ t v_2\|}{1-t}\|x_2- \frac{t}{\|(1-t)x_1+ t v_2\|}v_2\|\geq \frac{\|(1-t)x_1+ t v_2\|}{1-t}$, since $v_2 \in {x_1}^-$. Therefore, for all $\lambda \in \mathbb{R}$, $x_1 + \lambda v_2= \frac{\|(1-t)x_1+ t v_2\|}{1-t}x_2- \frac{t}{1-t}v_2 + \lambda v_2 \Rightarrow \inf_{\lambda \in \mathbb{R}}\|x_1 + \lambda v_2\|= \frac{\|(1-t)x_1+ t v_2\|}{1-t}\inf_{\lambda \in \mathbb{R}}\|x_2 + \lambda v_2\| \leq \inf_{\lambda \in \mathbb{R}}\|x_2 + \lambda v_2\|\Rightarrow \sqrt{1- {\epsilon_1}^2} \leq \sqrt{1- {\epsilon_2}^2}$. Therefore, $\epsilon_1= \epsilon_2$. Similarly, if we assume that $x_2= \frac{(1-t)x_1-tv_2}{\|(1-t)x_1-tv_2\|}, $ for some $0<t<1$, then we can apply similar arguments to prove that  $ \epsilon_1= \epsilon_2$.  \\
This completes the proof of the theorem.
\end{proof}
\begin{remark}
For $ \epsilon = 0 ,$   $F(x_1,0)= F(x_2,0) \Rightarrow x_1 = \pm x_2 $ holds if the space is strictly convex . Suppose    $F(x_1,0)= F(x_2,0). $ Then there exists $z $ such that $ x_1 \bot_B z $ and $ x_2 \bot_B z.$ Being strictly convex, Birkhoff-James orthogonality is left unique and so $ x_1 = \pm x_2.$ 
\end{remark}
\begin{remark}
It follows from Theorem $ 2.4 $ and Theorem $ 2.5 $ that given any normal cone $ K $ in a two-dimensional smooth Banach space $ \mathbb{X}, $ there exists a unique (upto multiplication by $ \pm 1 $) $ x \in S_{\mathbb{X}} $ and a unique $ \epsilon \in [0, 1) $ such that $ F(x, \epsilon)=K \cup (-K). $ This observation strengthens Theorem $ 2.4 $ to a considerable extent.
\end{remark}

Now we are in a position to prove the promised uniqueness theorem for $ F(x, \epsilon) $ in any normed space.

\begin{theorem}\label{theorem:F(x_1,e_1)=F(x_2,e_2)2}
Let $\mathbb{X}$ be a normed space. Let $x_1, x_2 \in S_{\mathbb{X}}$ and $\epsilon_1, \epsilon_2 \in (0,1)$. If $F(x_1, \epsilon_1)= F(x_2, \epsilon_2)$ then $x_1 = \pm x_2$ and $\epsilon_1= \epsilon_2$.
\end{theorem}
\begin{proof}
The proof of the theorem essentially follows from Theorem \ref{theorem:any dimension} and Theorem \ref{theorem:F(x_1,e_1)=F(x_2,e_2)}.
\end{proof}

Let us now shift our attention to the structure of $ G(x, \epsilon) $ in normed spaces. The following two lemmas are required to obtain the desired characterization.

\begin{lemma}\label{lemma:connected}
Let $\mathbb{X}$ be a two-dimensional Banach space. Let $(\theta \neq)z \in  \mathbb{X}, ~\epsilon \in [0,1)$.  Then $\overline{B}(z, \epsilon) \cap S_{\mathbb{X}}$ is either empty or path connected. 
\end{lemma}
\begin{proof}
 Note that $ \overline{B}(z, \epsilon) \cap S_{\mathbb{X}} = \emptyset$  if and only if $\| z - \frac{z}{\|z\|} \| > \epsilon. $  Assume $ \overline{B}(z, \epsilon) \cap S_{\mathbb{X}} \neq \emptyset.$ Let $ z_0 = \frac{z}{\|z\|}.$ Then $ z_0 \in \overline{B}(z, \epsilon) \cap S_{\mathbb{X}}.$

 Suppose that $y \in \overline{B}( z, \epsilon) \cap S_{\mathbb{X}}$. We show that there exists a path in $\overline{B}(z, \epsilon) \cap S_{\mathbb{X}}$ joining $y$ and $z_0$. Let $u= \frac{ z + y}{2}$. Then $\| z - u\|= \|y- u\|\leq \frac{\epsilon}{2}$. Let $v_t = \frac{(1-t)y + t z}{\|(1-t)y + t z\|}$, where $t \in [0,1]$. Claim that $ v_t \in \overline{B}( z, \epsilon) \cap S_{\mathbb{X}} $. 
Let $ v_t' = (1-t)y + t z.$ Then either $ v_t'= (1-r)  z + r u $ for some $r \in [0,1] $ or  $v_t'= (1-s) u + s y $ for some $ s \in [0,1].$ Assume $v_t'= (1-r)  z + r u$.  Then 
\begin{eqnarray*} 
  \| z- v_t' \| + \|v_t'-u\|& = & \| z- (1-r) z -r u\|+ \|(1-r) z + r u -u\| \\
								                           & = & r \| z -u\|+ (1-r)\| z - u\| \\
																					 & = & \| z - u\| \\
																					 & \leq & \frac{\epsilon}{2}. 
\end{eqnarray*}
Now, 
\begin{eqnarray*}
\|v_t-v_t'\| & = & \|v_t- \|v_t'\|v_t\|\\
             & =& |1-\|v_t'\||\\
						 & = & |\|y\|- \|v_t'\|| \\
						 & \leq & \|y-v_t'\| \\
						 & \leq & \|y-u\|+\|u-v_t'\| \\
						 & \leq & \frac{\epsilon}{2}+ \frac{\epsilon}{2}- \| z -v_t'\| \\
						\Longrightarrow \| z -v_t\|\leq \| z- v_t'\|+ \|v_t-v_t'\|\leq \epsilon.
\end{eqnarray*}
Hence $v_t \in \overline{B}(z, \epsilon) \cap S_{\mathbb{X}}$. Next assume that $v_t'= (1-s) u + s y$ for some $s \in [0,1].$  Then
\begin{eqnarray*}
\|y-v_t'\|+\|v_t'-u\|& =& \|y-(1-s)u-s y\|+\|(1-s)u+ s y -u\|\\
                     & =& (1-s)\|y-u\|+ s\|y-u\|\\
										&=&\|y-u\|\\
										&\leq & \frac{\epsilon}{2}.
\end{eqnarray*}
Again, 
\begin{eqnarray*} 
\|v_t-v_t'\|&=& \|v_t- \|v_t'\|v_t\|\\
            &=&|1-\|v_t'\||\\
						&=& |\|y\|- \|v_t'\||\\
						&\leq &\|y-v_t'\|.
\end{eqnarray*}
Hence 
\begin{eqnarray*}
\| z - v_t\|& \leq & \| z - u\|+ \|u-v_t\|\\
            &\leq & \frac{\epsilon}{2}+ \|u-v_t'\|+\|v_t'-v_t\|\\
						&\leq &\frac{\epsilon}{2}+ \|u-v_t'\|+\|v_t'-y\|\\
						&\leq & \epsilon.
\end{eqnarray*}
Thus, $ v_t \in \overline{B}(z, \epsilon) \cap S_{\mathbb{X}}$.
Consider the map $ f : [0,1] \longrightarrow \overline{B}( z, \epsilon) \cap S_{\mathbb{X}} $ by $f(t) = v_t.$  Then $f$ is a path in $\overline{B}( z, \epsilon) \cap S_{\mathbb{X}}$ joining $f(0) = y$ and $f(1) = z_0.$ Hence  $\overline{B}( z, \epsilon) \cap S_{\mathbb{X}}$ is path connected.   This proves the lemma.   
\end{proof}

\begin{lemma}\label{lemma:chem}
Let $\mathbb{X}$ be a two-dimensional Banach space. Let $z \in S_{\mathbb{X}}$ and $\epsilon \in [0,1)$. Then $A= [\bigcup_{\alpha \in \mathbb{R}}\overline{B}(\alpha z, \epsilon)] \cap S_{\mathbb{X}}$ has at most two components.
\end{lemma}
\begin{proof}
If $0 \leq \alpha < 1- \epsilon$, then for any $w \in S_{\mathbb{X}}$, we have, $\|\alpha z- w\|\geq \|w\|- \|\alpha z \|= 1- \alpha > \epsilon$. If $\alpha > 1+ \epsilon$, then for any $w \in S_{\mathbb{X}}$, $\|\alpha z -w\|\geq \|\alpha z\|-\|w\|=\alpha -1 > \epsilon$. Similarly, if $-1+ \epsilon < \alpha \leq 0$ or $\alpha < -1- \epsilon$, then for any $w \in S_{\mathbb{X}}$, we have, $\|\alpha z - w\|> \epsilon$. Therefore,
\[ A= [\{\bigcup_{\alpha \in [1- \epsilon,1+ \epsilon]} \overline{B}(\alpha z , \epsilon)\}\cap S_{\mathbb{X}}]\bigcup [\{\bigcup_{\alpha \in [-1- \epsilon,-1+ \epsilon]} \overline{B}(\alpha z , \epsilon)\}\cap S_{\mathbb{X}}].\] Let 
\[A_1=[\{\bigcup_{\alpha \in [1- \epsilon,1+ \epsilon]} \overline{B}(\alpha z , \epsilon)\}\cap S_{\mathbb{X}}], \]
\[A_2=[\{\bigcup_{\alpha \in [-1- \epsilon,-1+ \epsilon]} \overline{B}(\alpha z , \epsilon)\}\cap S_{\mathbb{X}}].\] Clearly, to prove the lemma, it is sufficient to show that both $A_1,A_2$ are connected. Now, $A_1=[\{\bigcup_{\alpha \in [1- \epsilon,1+ \epsilon]} \overline{B}(\alpha z , \epsilon)\}\cap S_{\mathbb{X}}]$. It is easy to see that for each $\alpha \in [1-\epsilon,1+\epsilon]$, $z \in \overline{B}(\alpha z, \epsilon)\cap S_{\mathbb{X}}$. Also, by Lemma \ref{lemma:connected}, $\overline{B}(\alpha z, \epsilon)\cap S_{\mathbb{X}}$ is path connected for each $\alpha \in [1- \epsilon,1+\epsilon]$. Since union of connected sets having nonempty intersection is connected, $A_1$ is connected. Similarly, we can show that $A_2$ is connected. This completes the proof of the lemma.   
\end{proof}

Let us now exhibit the interconnection between $ G(x, \epsilon) $ and normal cones in a two-dimensional Banach space.
\begin{theorem}\label{theorem:chem}
Let $\mathbb{X}$ be a two-dimensional Banach space. Let $x \in S_{\mathbb{X}}$ be a smooth point and $\epsilon \in [0,1)$. Then there exists a normal cone $K$ in $\mathbb{X}$ such that $G(x, \epsilon)= K \cup (-K)$. 
\end{theorem}
\begin{proof}
It follows from the characterization of $\epsilon-$Birkhoff-James orthogonality given by Chmieli\'nski et. al.in \cite{CSW} that
\[ G(x,\epsilon)\cap S_{\mathbb{X}}=\{y\in S_{\mathbb{X}}:\|\alpha z - y\| \leq \epsilon ~\forall \alpha \in \mathbb{R}\},\]
 where $z \in S_{\mathbb{X}}$ is the unique (up to multiplication by $\pm 1$) vector such that $x \perp_B z$, since $x \in S_{\mathbb{X}}$ is smooth. Therefore, $G(x, \epsilon)\cap S_{\mathbb{X}}= [\bigcup_{\alpha \in \mathbb{R}}\overline{B}(\alpha z, \epsilon)] \cap S_{\mathbb{X}}$. Applying Lemma \ref{lemma:chem}, we see that $G(x, \epsilon)\cap S_{\mathbb{X}}$ has at most two connected components. However, it is easy to observe that $\pm z \in G(x, \epsilon )\cap S_{\mathbb{X}} $ and $\pm x \notin G(x, \epsilon )\cap S_{\mathbb{X}}$. Therefore, $G(x, \epsilon )\cap S_{\mathbb{X}}$ has exactly two connected components, one of them containing $z$ and the other containing $-z$. Since $G(x, \epsilon )\cap S_{\mathbb{X}}$ is closed, each of the components must be closed in $S_{\mathbb{X}}$. It also follows from the property of $G(x, \epsilon)$ that the two components of $G(x, \epsilon)$ are symmetric with respect to origin. Since $\mathbb{X}$ is a two-dimensional Banach space, any connected component of $S_{\mathbb{X}}$, not containing two antipodal points, must be an arc of $S_{\mathbb{X}}$. Therefore, there exists $v_1, v_2 \in S_{\mathbb{X}}$ such that
 \[G(x, \epsilon)\cap S_{\mathbb{X}} =\{\frac{(1-t)v_1+ t v_2}{\|(1-t)v_1 + t v_2\|}: t \in [0,1]\} \cup \{\frac{-(1-t)v_1- t v_2}{\|-(1-t)v_1 - t v_2\|}: t \in [0,1]\}.\]
 Let $K$ be the normal cone determined by $v_1, v_2$. Then from the properties of $G(x, \epsilon),$ it is clear that $G(x, \epsilon)= K \cup (-K)$. This completes the proof of the theorem. 
\end{proof}

Let $ Q_{x, y}(\epsilon) $ denote the restriction of  $ G(x, \epsilon) $ to the subspace spanned by $x$ and $y$. The following theorem, that describes the structure of $ G(x, \epsilon) $ in any normed space, is immediate. 
\begin{theorem}\label{theorem:any dimension chem}
Let $\mathbb{X}$ be a normed space. Let $x \in \mathbb{X}$  be a smooth point and $\epsilon \in [0,1)$. Then $G(x, \epsilon)= \bigcup_{y\in \mathbb{X}}Q_{x,y}(\epsilon)$. In particular, $ G(x, \epsilon) $ is a union of two-dimensional normal cones.
\end{theorem}

\begin{remark}

Theorem $ 2.3 $ and Theorem $ 2.8 $ together imply that in a normed space, both $ F(x, \epsilon) $ and $ G(x, \epsilon) $ are unions of two-dimensional normal cones. However, since these two different types of  approximate Birkhoff-James orthogonality do not coincide in a general normed space, it follows that the constituent two-dimensional normal cones for $ F(x, \epsilon) $ and $ G(x, \epsilon) $ may not be identical. 

\end{remark}

\begin{acknowledgement}
The authors would like to heartily thank Professor Vladimir Kadets for his extremely insightful input towards proving Theorem 2.7.\\
 The first author feels elated to acknowledge the blissful presence of Mr. Arunava Chatterjee, his childhood friend and a brilliant strategist, in every sphere of his life.
\end{acknowledgement}

\end{document}